\documentclass[
  a4paper, 
  reqno, 
  oneside, 
  11pt
]{amsart}

\usepackage[utf8]{inputenc}

\usepackage[dvipsnames]{xcolor} 
\colorlet{cite}{LimeGreen!50!Green}
\usepackage{tikz}  
\usetikzlibrary{arrows,positioning}
\tikzset{ 
  baseline=-2.3pt,
  text height=1.5ex, text depth=0.25ex,
  >=stealth,
  node distance=2cm,
  mid/.style={fill=white,inner sep=2.5pt},
}

\usepackage{lmodern}
\usepackage{tikz-cd}
\usepackage{amsthm, amssymb, amsfonts}

\usepackage{graphicx,caption,subcaption}
\usepackage{braket}  
\usepackage{microtype}

\usepackage[%
  bookmarks=true,			
  unicode=true,			
  pdftitle={Deformations of noncompact manifolds},		%
  pdfauthor={Elizabeth Gasparim}{Francisco Rubilar},	%
  pdfkeywords={Hodge structures}{adjoint orbit}{compactification}{Lefschetz fibration},	
  colorlinks=true,		
  linkcolor=Blue,			
  citecolor=cite,		
  filecolor=magenta,		
  urlcolor=RoyalBlue			
]{hyperref}				
\usepackage{cleveref}

\newtheoremstyle{mydef}
  {}		
  {}		
  {}		
  {}		
  {\scshape}	
  {. }		
  { }		
  {\thmname{#1}\thmnumber{ #2}\thmnote{ #3}}	

\theoremstyle{plain}	
\newtheorem{theorem}{Theorem} 

\newtheorem*{theorem*}{Theorem}

\newtheorem{corollary}[theorem]{Corollary}
\newtheorem*{corollary*}{Corollary}

\newtheorem*{proposition*}{Proposition}
\newtheorem{proposition}{Proposition}

\newtheorem*{lemma*}{Lemma}
\newtheorem{lemma}{Lemma}

\theoremstyle{mydef} 
\newtheorem{definition}{Definition}
\newtheorem*{conjecture*}{Conjecture}

\theoremstyle{remark}
\newtheorem{remark}{Remark}
\newtheorem{notation}{Notation}
\newtheorem{example}{Example}

\theoremstyle{definition}
\theoremstyle{remark}

\DeclareMathOperator{\Pic}{Pic}

\DeclareMathOperator{\End}{End}

\DeclareMathOperator{\rk}{rk}

\DeclareMathOperator{\red}{red}
\DeclareMathOperator{\reg}{reg}

\DeclareMathOperator{\irreg}{irreg}

\newcommand{\ce}{\mathrel{\mathop:}=}

\begin{document}
\author{{\bf Edoardo Ballico and Elizabeth Gasparim}\\ \\
{\tiny \em Univ. of Trento, Italy, \& Univ. Catolica del Norte, Chile} \\ 
\tiny{edoardo.ballico@unitn.it, \& etgasparim@gmail.com}}

\address{}
\email{}
\title[Irregular bundles on Hopf surfaces]{Irregular bundles on Hopf surfaces}


\begin{abstract}
We discuss the hypersurfaces of the moduli spaces of rank $2$ vector bundles 
on a classical Hopf surface formed by irregular bundles.
\end{abstract}

\maketitle

\textbf{Keywords:} Moduli spaces, vector bundles, Hopf surfaces.

\tableofcontents

\section{Statements of results}
Let $\pi\colon X\to \mathbb P^1$ be a classical Hopf surface  with $T$ as its elliptic fiber. 
  For all $n>0$, the moduli space $\mathcal M_n$ 
  of $SL(2,\mathbb C)$  holomorphic bundles with $c_2=n$ 
  on $X$ contains an open dense subset $\mathcal M_n(0)$ formed by bundles without jumps. 
  The goal of this note is to describe some of the geometric features of 
  $\mathcal M_n(0)$. 
  \\

A bundle $E \in \mathcal M_n$ is called 
  irregular when its restriction to any fiber $T$ of $X$ has automorphism group of dimension 4
  (Def.\thinspace \ref{irregdef}), this happens whenever the restriction of $E\vert_T= L\oplus L$ 
  with $L^2=\mathcal O$, that is when $L$ is one of the 4 thetas. 
  We then compare regular and irregular bundles without jumps. We first show:

\begin{theorem*}[\ref{aa1}]
Fix $E\in \mathcal M_n(0)$, $n\ge 2$. 
\begin{itemize}
\item[(a)] If $E\in \mathcal M_n^{\reg}(0)$, then $C_E$ is smooth and connected.
\item[(b)] If $E\in \mathcal M_n^{\irreg}(0)$, then $C_E$ is singular and irreducible.
\end{itemize}
\end{theorem*}
 Accordingly,  
  we define a concept of weight of a bundle
(Def.\thinspace \ref{weights}) 
$$w(E) = w_1(E)+w_2(E)+w_3(E)+w_4(E)$$
intended as a measure 
of irregularity, and prove:
\begin{theorem*}[\ref{irregw}] $E\in \mathcal M_n(0)$ is irregular if and only if $w(E)>0$.
\end{theorem*}

Having identified  $\mathcal M_n^{\irreg}(0)$ as the set of all bundles $E \in \mathcal M_n^{sg}$ without jumps and  
having singular spectral curve $C_E$ (Not.\thinspace \ref{nota}), we prove:

\begin{theorem*}[\ref{Wrr5}]
Assume $n>1$. For all $E\in \mathcal M_n^{\irreg}(0)$  we have:
\begin{enumerate}
\item $w(E)\le 2n-2$.
\item $E$ has $(n, 0,\dots, 0)$ has its $i$-th profile if and only if $w_i(E)\le n-1$ for all $i$ and $w_i(E)=n-1$ .
\item Fix $i\in \{1,2,3,4\}$, a positive integer $\ell$ and a restricted profile $(m_1,\dots ,m_\ell)$. 
There exists a non-empty locally closed analytic subset $\Gamma \in \mathcal M_n^{\irreg}(0)$
such that  all  $G(E)\in |\mathcal O_{\mathbb P^1\times \mathbb P^1}(n,1)|$ have   $(m_1,\dots ,m_\ell)$ as $i$-th profile
and this set of graphs $G(E)$  has codimension at most $-\ell +m_1+\cdots +m_\ell$ in $|\mathcal O_{\mathbb P^1\times \mathbb P^1}(n,1)|$.
\end{enumerate}
\end{theorem*}

\begin{theorem*}[\ref{Wrr7}]
Fix an ordering $i_1,i_2,i_3,i_4$ of the set $\{1,2,3,4\}$.
\begin{itemize}
\item[(a)] For all $n\ge 2$ there exists  a codimension $2$ locally closed analytic subset $\Gamma \subset \mathcal M_n(0)$ such that $w_{i_1}(E)=w_{i_2}(E)=1$ and $w_{i_3}(E)=w_{i_4}(E)=0$ for all $E\in \Gamma$.

\item[(b)] For all $n\ge 4$ there exists  a codimension $3$ locally closed analytic subset $\Delta \subset \mathcal M_n(0)$ such that $w_{i_1}(E)=w_{i_2}(E)=w_{i_3}(E)=1$ and $w_{i_4}(E)=0$ for all $E\in \Delta$.
\end{itemize}
\end{theorem*}

\begin{theorem*}[\ref{rr9}]
Fix an integer $n\ge 2$. Fix an ordering $i_1,i_2,i_3,i_4$ of $\{1,2,3,4\}$. Then there exists an irreducible locally closed analytic subset of dimension $2n+2$ of $\mathcal M_n(0)$
such that for all $E\in \Gamma$ the restricted $i_1$-th and $i_2$-th profiles of $E$ are $(n)$, while $w_{i_3}(E)=w_{i_4}(E)=0$.
\end{theorem*}

We then describe some features of the topology of $\mathcal M_n^{\reg}(0)$, obtaining:

\begin{theorem*}[\ref{full1bis}]
Fix an integer $n\ge 2$. We have:
\begin{itemize}
\item[(i)]   $H^*(\mathcal M_n^{\reg}(0),\mathbb Z)= H^*(\mathcal A_n,\mathbb Z)\otimes H^*((S^1)^{4n-2},\mathbb Z),$
and 
\item[(ii)]   $H^i(\mathcal M_n^{\reg}(0),\mathbb Z)=0$ \ for all $i\ge 6n$.
\end{itemize}
\end{theorem*}

\begin{theorem*}[\ref{ttt1}]
Let $Y\subset \mathcal M_n(0)$ be an irreducible compact complex analytic subspace. Then $G(Y)$ is a single point.
\end{theorem*}

\begin{corollary*}[\ref{ttt2}]
Let $Y\subset \mathcal M_n^{\reg}(0)$ be an irreducible compact complex analytic subspace. Then $Y$ is contained in one of the $(2n-1)$-dimensional tori
which are the fibers of the graph map.
\end{corollary*}

One might also consider the moduli spaces $\mathcal M_{n, \delta}$ of bundles on $X$ having 
determinant $\delta$ in place of trivial determinant. 
We conclude the paper by showing (Prop.\thinspace \ref{wpre2}) that $\mathcal M_{n, \delta}$ and $\mathcal M_n$ are biholomorphic, 
therefore all results proved here have analogous statements for the cases of nontrivial determinant. 

Our interest in bundles on Hopf surfaces originated from the survey of E. Witten 
about instantons on $S^3 \times S^1$
\cite{W}.

\begin{remark}
The real 4-manifold $S^3 \times S^1$ admits the structure of an elliptically fibered
 complex surface diffeomorphic to  a classical Hopf surface, 
which we denote by $X$ in this work. 
Using \cite[Thm.\thinspace 1]{Bu}, the moduli space of $SU(2)$ instantons on $S^3 \times S^1$
of charge $n$ can be identified with the moduli space of stable $SL(2,\mathbb C)$ bundles on $X$ with $c_2=n$.
 Here we denote by $\mathcal M_n$  the moduli space of  stable rank 2 vector bundles on 
 a classical Hopf surface $X$
 with $c_2=n$ and trivial determinant. 
  In the literature about the gauge theory of instantons on $S^4$ and  on ruled surfaces, 
  strong homological stability 
  of moduli spaces were proven using the description of jumping lines \cite{BHMM, HM}
  and of local contributions to instanton charges concentrated at exceptional curves \cite{Ga}.
It is now know
  whether the Atiyah--Jones conjecture \cite{AJ} holds true over $S^3\times S^1$.
  For the case of instantons on $S^3\times S^1$, the corresponding complex geometry
  presents a  completely new phenomenon.
 Namely, one  finds out that 
  the second Chern classes of vector bundles on $X$ can not be calculates via  the standard
  { jumping fibers} techniques used for rational surfaces. The reason is that 
   for all $n>0$, the moduli space $\mathcal M_n$ 
  of $SL(2,\mathbb C)$  holomorphic bundles 
  on $X$ contains a dense open (with analytic complement) subset $\mathcal M_n(0)$  formed by bundles without any jumps.
  \end{remark}
 
  Here we concentrate in discussing geometric aspects of the 
   set $\mathcal M_n(0)$ of bundles without jumps, by dividing it into the subsets of regular 
   and irregular bundles.  
  The second author discusses connectedness of 
  $\mathcal M_n$ in \cite{Ga2}, but we do not use that result here.

\section{Bundles without jumps}
Let $\pi\colon X\to \mathbb P^1$ be a classical Hopf surface  with $T$ as its elliptic fiber. 
We have $\Pic^0(T)$ (often written as $T^\ast$) which is non-canonically isomorphic to $T$.
Observe that, since $T$ is a torus, there are $4$ such elements $L\in \mathrm{Pic}^0(T) $ 
satisfying $L^{\otimes 2} \cong \mathcal O_T$, these are called half-periods. 
Each of them is informally referred to as a \lq\lq theta\rq\rq (as in theta characteristic), and the expression 
\lq\lq there are $4$ thetas\rq\rq \,  is commonly used. 
Hence, the 4 thetas of $T$ are the 4 elements $L$ of $\mathrm{Pic}^0(T)$  such that $L^{\otimes 2}\cong \mathcal O_T$. 
Informally, thinking of the torus as obtained from $\mathbb C^2/\Lambda$ where $\Lambda$ is the 
integer lattice, these can be though of as 
corresponding to the points  $(0,0), (\frac{1}{2},0), (0, \frac{1}{2}), (\frac{1}{2},\frac{1}{2})$.\\

Let $F$ be a rank $2$ vector bundle on $T$
with $\det(F)\cong \mathcal O_T$. If $F$ is not semistable, then $F\cong R\oplus R^\vee$ for some line bundle $R$ on $T$ with $\deg(R) \ne 0$.
Now assume that $F$ is semistable and decomposable, i.e. assume $F\cong L\oplus L^\vee$ for some degree $0$ line bundle on $T$, then there are 2 possibilities:
\begin{itemize}
\item If $L\ne L^\ast$, then $\dim \End(F) =2$. 
\item If $L\cong L^\ast$, then $\dim \End(F)=4$.
\end{itemize}

\begin{notation}\label{irregdef} 
A  bundle $E\in \mathcal M_n$  is called {\bf regular} it its restriction to every fiber of $\pi$ has 
automorphism group of minimal dimension, i.e. $\dim \mbox{Aut}(E\vert_{T}) =  2\, \forall T$.
Note that, for a fiber $T$, we have $L=L^\ast$ if and only if $L^{\oplus 2}= \mathcal O_T$, i.e. if and only if $L$ is a theta;
 in which case the bundle is called {\bf irregular} over $T$. 
 A bundle $E$ over $X$ is called {\bf irregular} if its restriction to any fiber of $\pi$ is irregular. \\
\end{notation}

Let $\mathcal M_n$ denote the moduli space
of all rank $2$ stable vector bundles on $X$ with trivial determinant and $c_2(E)=n$, i.e.
$$\mathcal M_n=\{E\,  \text{stable  bundle on} \, X,   \rk(E)=2, \det(E)\cong \mathcal O_X, c_2(E)=n\}.$$

The variety $\mathcal M_n$ is  smooth non-empty    of dimension 
$4n$ \cite[Prop.\thinspace 3.4.4]{BH}, see also 
\cite[Prop.\thinspace 4.2]{BMo} for the case of nontrivial determinant.
Since here we consider only the case of trivial determinant,  the 
{\bf spectral involution} on $T^*$ may be defined 
as $i(\lambda) = -\lambda$, with corresponding {\bf double cover} 
\begin{equation}\label{dc}\mathbb P^1 \times T^* \to \mathbb P^1 \times \mathbb P^1\end{equation} 
$$(b, \lambda) \stackrel{\sigma}{\mapsto} (b, \{\lambda, i(\lambda)\}).$$

Given a rank 2 bundle $E$ on $X$, let $V$ be the universal Poincaré line bundle on
$X \times \mathbb C^*$ and consider the first derived image $R^1\pi'_*(E \otimes V)$ where
$\pi'$  the projection $\pi'= \pi\times id \colon X\times \mathbb C^* \to \mathbb P^1 \times \mathbb C^*$. 
The sheaf $R^1\pi'_*(E \otimes V)$ is supported on a divisor on $ \mathbb P^1\times \mathbb C^*$, 
which descends to a divisor 
$$D \subset \mathbb P^1\times \mathbb P^1 = \pi(X) \times \Pic^0(T)/\pm1.$$ When $c_2(E)=n$, the divisor
 $G(E) \ce D $ belongs to the linear system $|\mathcal O(n,1)|$ over $\mathbb P^1\times \mathbb P^1$
 and is called the {\bf graph} of $E$, see\cite[Prop.\thinspace 3.2.3]{BH}. \\

Let  $\mathcal M_n(0)$ denote the set of all $E\in  \mathcal M_n$ with {\bf no jumps}, i.e.
$$\mathcal M_n(0) \ce  \{E\in  \mathcal M_n,  \forall T_x \, E\vert_{T_x}\, \text{ is semistable}\}.$$

 Let  $\mathcal M_n^{\reg}(0)$ denote the subset of all {\bf regular bundles} in $\mathcal M_n(0)$, that is, 
 those whose restrictions to fibers are semistable but not sums of  half-periods, i.e.
 $$\mathcal M_n^{\reg}(0)= \{E\in  \mathcal M_n: \forall   T_x, 
  E\vert_{T_x}   \not\cong  \, L_0\oplus i^\ast(L_0) \, \text{with}\, L_0^{\otimes 2} \cong \mathcal O_{T_x} \}.$$

\begin{remark}The image $G(\mathcal M_n(0))$ of $\mathcal M_n(0)$ is the Zariski open set of all graphs $D\in |\mathcal O_{\mathbb P^1\times \mathbb P^1}(n,1)|$ which are smooth (or, equivalently in this case, irreducible), i.e. the set of all $D\in |\mathcal O_{\mathbb P^1\times \mathbb P^1}(n,1)|$ 
giving a degree $n$ morphism $\mathbb P^1\to \mathbb P^1$.
\end{remark}

Since $\mathcal M_{\delta,1}$ is well-understood, from now on we consider $n>1$.

Assume $n>1$ and set $$\mathcal M_n^{\irreg}(0)\ce \mathcal M_n(0)\setminus \mathcal M^{\reg}_{\delta,n}(0).$$ 
Since the notions of irregularity and of having no jumps are completely 
 described in terms of graphs, we have the equalities
$$ \mathcal M_n(0) =G^{-1}(G(\mathcal M_n(0))), \quad \quad 
 \mathcal M_n^{\reg}(0) =G^{-1}(G(\mathcal M_n^{\reg}(0))),  $$and
$$ \mathcal M_n^{\irreg}(0) =G^{-1}(G(\mathcal M_n^{\irreg}(0))).$$

\begin{remark}\label{rr1}
The set $G(\mathcal M_n\setminus \mathcal M_n(0))$ is an irreducible quasi-projective variety of dimension $2n$, i.e. of codimension $1$ in $\mathbb P^{2n+1}$. It consists of singular graphs.
\end{remark}

\begin{remark}
 The set $G(\mathcal M_n^{\irreg}(0))$ is a quasi-projective variety  and each irreducible component of it has dimension $2n$, i.e. it has codimension $1$ in $\mathbb P^{2n+1}$. Since $G$ is surjective and since there are $4$ thetas, we get $4$ different \lq\lq types\rq\rq of irregular bundles, one for each of the $4$ thetas.
\end{remark}

For each $A\in G(\mathcal M_n^{\reg}(0))$ the fiber $G^{-1}(A)$ is well-understood, it is isomorphic to the Jacobian of the spectral curve of $A$ and in particular
it is a smooth and connected projective manifold of dimension $2n-1$.  

Recall that for $n>1$ 
the graph map $G\colon\mathcal M_n\to \mathbb P^{2n+1}$ is surjective, that $\mathbb U:= G(\mathcal M_n^{\reg}(0))$  is a Zariski open subset of $\mathbb P^{2n+1}$ and that for each $a\in \mathbb U$ the set $G^{-1}(a)$ is isomorphic to an Abelian variety of dimension $2n-1$.

\section{Regular v irregular}\label{irregVreg}

Let $\mathcal M_n(0)$ be the set of all $E\in \mathcal M_n$ with no jumps, then all elements 
$E\in \mathcal M_n(0)$ such that $G(E)$ is a smooth element of $|\mathcal O_{\mathbb P^1\times \mathbb P^1}(n,1)|$.
Set $$\mathcal M_n(\ge 1):= \mathcal M_n\setminus \mathcal M_n(0).$$ 
Since being singular is a closed analytic condition, $\mathcal M_n(\ge 1)$ is a closed analytic subset of $\mathcal M_n$, while $\mathcal M_n(0)$ is an open subset of $\mathcal M_n$. It is easy to check  that 
$\mathcal M_n(\ge 1)$ is the set of all $E\in \mathcal M_n$ with at least one jump.

Recall from notation \ref{nota} that  $\mathcal M_n^{sg}$ is the set of all $E\in \mathcal M_n(0)$ such that 
the spectral curve $C_E$ is singular, while $\mathcal M_n^{sm}$ is the set of all $E\in \mathcal M_n(0)$ 
such that the spectral curve $C_E$ of $E$ is smooth.

Note that $\mathcal M_n(0)$ is set of all $E\in \mathcal M_n$ such that $G(E)$ has no vertical component, i.e. it is an irreducible element of $|\mathcal O_{\mathbb P^1\times \mathbb P^1}(n,1)|$. We have used the notation 
$|\mathcal O(n,1)|_{sm}$ for  the set of such irreducible elements. Observe that  in both cases  
 $G(\mathcal M_n^{\reg}(0)) \subset |\mathcal O(n,1)|_{sm}$ and 
$G(\mathcal M_n^{\irreg}(0)) \subset |\mathcal O(n,1)|_{sm}$.

For any $E\in \mathcal M_n(0)$ let $C_E$ denote its spectral curve.
Let $i$ denote the involution $L\mapsto L^*$ on $T^*$. We have $T^*/i\cong \mathbb P^1$ and the quotient map $T\to \mathbb P^1$ is ramified over $4$ points, $a_1,a_2,a_3,a_4$, with are the images of of the $4$ thetas of $T^*$. We get $4$ element $\mathbb P^1\times \{a_i\}\in |\mathcal O_{\mathbb P^1\times \mathbb P^1}(0,1)|$, $i=1,2,3,4$.

Observe that  $\mathcal M_n^{\reg}(0)$ is  the set of all $E\in \mathcal M_n(0)$ such that $G(E)$ is transversal to each $\mathbb P^1\times \{a_i\}$, for $i=1,2,3,4$, i.e. if $E\in \mathcal M_n^{\reg}(0)$ then
$G(E)$ intersects $\Sigma =\mathbb P^1\times \{a_1,a_2,a_3,a_4\}$ at points. 
 Observe also that $\mathcal M_n^{\irreg}(0)= \mathcal M_n(0)\setminus \mathcal M_n^{\reg}(0)$, and that
all that elements of $\mathcal M_n^{\irreg}(0)$ have irreducible graphs.

\begin{theorem}\label{aa1}
Fix $E\in \mathcal M_n(0)$, $n\ge 2$. 
\begin{itemize}
\item[(a)] If $E\in \mathcal M_n^{\reg}(0)$, then $C_E$ is smooth and connected.
\item[(b)] If $E\in \mathcal M_n^{\irreg}(0)$, then $C_E$ is singular and irreducible.
\end{itemize}
\end{theorem}

\begin{proof}
By assumption $G(E)\cong \mathbb P^1$. By definition $u: C_E\to G(E)$ is a double covering ramified exactly 
at the $4n$  points of $H_i\cap G(E)$ with $i=1,2,3,4$.
Hence $C_E\setminus u^{-1}(G(E))$ is smooth and it is a 2 to 1 unramified covering. 
This covering comes from a covering
 $u'\colon \mathbb P^1\times T^*\to \mathbb P^1\times T^*/i$ 
 which is ramified exactly over $\Sigma$. We get a universal double covering over all smooth curves. 
 We get that each $C_E$ is a certain effective divisor of $\mathbb P^1\times T^*$ with arithmetic genus $2n-1$. Hence $C_E$ is smooth if and only if it is over its points which are mapped bijectively onto the points of $G(E)\cap \Sigma$ and this is the case if and only if $E\in \mathcal M_n^{\reg}(0)$.
Since $G(E)\cap \Sigma$ we get that $C_E$ is connected. 
Hence we get (a).

For part (b), assume that $C_E$ is not irreducible. Since $G(E)\cong \mathbb{P}^1$ and $C_E\to G(E)$ is a degree $2$ morphism, $C_E$ has $2$ irreducible components, both of them smooth and isomorphic to $\mathbb{P}^1$, say $C_E =J_1\cup J_2$. Remember that $C_E\subset \mathbb{P}^1\times T^*$. Hence, $\pi_2(J_i)$ is a single point.
Hence $G(E)$, which is the quotient of $C_E$ by the involution, is not an element of 
$|\mathcal{O}_{\mathbb{P}^1\times \mathbb{P}^1}(n,1)|$, a contradiction.
Since $C_E$ is irreducible  of arithmetic genus $2n-1$
and the 2-to 1 map $C_E \to G(E)$ has at most $4n-1$ ramification points, $C_E$ is singular.
\end{proof}

\begin{lemma}\label{aa2}
If $G(E)\in |\mathcal O(n,1)|_{sm}$, then $E$ is regular.
\end{lemma}

\begin{proof}
Take $x\in \mathbb P^1$ such that $E_{|\pi^{-1}(x)}$ is an element of $T^*/i$ associated to a theta, i.e. it corresponds to a point $a\in \Sigma\cap G(E)$. We need to prove that $E_{|\pi^{-1}(x)}$ is indecomposable, i.e. that $h^1(\pi^{-1}(x),E_{|\pi^{-1}(x)})=1$. The transversality of $G(E)$ and $\Sigma$ implies that $C_E$ is smooth at its points, $a'$, over $a$ and that the torsion sheaf  $R^1$ on $\mathbb P^1\times T$ is locally free of rank $1$ at the point $a'$ as an $\mathcal O_{C_E}$-sheaf \cite[Remark 2.8]{T}.
\end{proof}

The following lemma is well-known, we just state it  for 
completeness.

\begin{lemma}\label{aa3}
Consider  a graph $A$ such that $A =G(E)$ for some $E\in \mathcal M_n^{\reg}(0)$. Then $G^{-1}(A)$ is isomorphic to the Jacobian $J(C_E)$ of the spectral curve $C_E$
and hence it is connected, smooth, compact and isomorphic to an Abelian variety of dimension $2n-1$.
\end{lemma}

\begin{proof}
By Lemma \ref{aa1} $C_E$ is a smooth and connected curve of genus $2n-1$. Hence its Jacobian $J(C_E)$ is connected, smooth, compact and isomorphic to an Abelian variety of dimension $2n-1$. Lemma \ref{aa2} says that any bundle with $A$ as its graph is regular.
In our case the base, $B$, of the Hopf fibration is $\mathbb P^1$. Since $J(\mathbb P^1)$ is a singleton, it is sufficient to quote \cite[Th. 5.14]{FMW}.
\end{proof}


\section{Irregular profiles}
We fix $\delta\in \mathrm{Pic}(X) =\mathrm{Pic}^0(X)\cong \mathbb C^\ast$. Let $I\subset T^\ast$ be the set of all thetas with respect to the involution $i_\delta$.
We we fix an order for the $4$ elements of $I$,  we write them as $L_1,L_2,L_3,L_4$, and write $a_1, a_2, a_3, a_4$ 
for the $4$ elements of $\mathbb P^1$ associated to them, see \eqref{dc}.
Let $\pi_2\colon \mathbb P^1\times \mathbb P^1\to \mathbb P^1$ denote the projection onto the second factor. 
Set 
$$D_i\ce \pi_2^{-1}(a_i).$$ We get $4$ elements of $|\mathcal O_{\mathbb P^1}(1,0)|$. 

Let  $|\mathcal O_{\mathbb P^1\times \mathbb P^1}(n,1)|_{sm}$ denote the set of smooth elements in $|\mathcal O_{\mathbb P^1\times \mathbb P^1}(n,1)|$.
 Note that $E\in \mathcal M_n(0)$, i.e. $E$ has no jumps, if and only if $G(E)\in |\mathcal O_{\mathbb P^1\times \mathbb P^1}(n,1)|_{sm}$. For each $U\in |\mathcal O_{\mathbb P^1\times \mathbb P^1}(n,1)|_{sm}$ we get $4$ degree $n$ zero-dimensional schemes
 $$Z_U(i)\ce U\cap D_i, \quad  i=1,2,3,4.$$
  For any $E\in \mathcal M_n(0)$, taking $G(E)$ in place of $U$,  we get $4$ zero-dimensional schemes $Z_{G(E)}(i)$, $i=1,2,3,4$.
  
  \begin{remark}
Note that $E$ is irregular if and only if at least one of the 4 schemes $Z_{G(E)}(i)$, $i=1,2,3,4$, is not reduced, i.e., it is not formed by $n$ distinct points.
\end{remark}

\begin{notation}\label{defH} For each $i=1,2,3,4$, denote by  $H_i$  the set of all divisors 
$D\in |\mathcal O_{\mathbb P^1\times \mathbb P^1}(n,1)|_{sm}$ such that $D$ is tangent to $D_i$ 
(the point of tangency with $D_i$ is not fixed), and let $\mathcal H_i\ce G^{-1}(H_i)$.
\end{notation}

Fix $p\in D_i$ and let $(2p,D_i)$ denote the degree $2$ effective divisor of $D_i$ with $p$ as its reduction. Note that $(2p,D_i)$ is a degree $2$ connected zero-dimensional scheme. Since $n\ge 1$, $\mathcal O_{\mathbb P^1\times\mathbb P^1}(1,1)$ is very ample and $\deg(2p,D_i)=2$, we have $$h^0(\mathbb P^1\times \mathbb P^1,\mathcal I_{(2p,D_i)}(n,1))=
h^0(\mathbb P^1\times \mathbb P^1,\mathcal O_{\mathbb P^1\times \mathbb P^1}(n,1))-2$$
 and a general $D\in |\mathcal I_{(2p,D_i)}(n,1)|$ is smooth. 
Varying the point $p\in D_i$ we get  a non-empty hypersurface $\tilde{H}_i$ of $|\mathcal O_{\mathbb P^1\times \mathbb P^1}(n,1)|$ (see Theorem \ref{Wrr5} and its proof for more details, it implies for instance that $\mathcal H_i\ne \mathcal H_j$ if $i\ne j$).
Observe that $$H_i\ce \tilde{H}_i\cap |\mathcal O_{\mathbb P^1\times \mathbb P^1}(n,1)|_{sm},$$ that is, 
  $H_i$ is the set of all smooth elements of $\tilde{H}_i$ (these are the same $H_i$ from Notation \ref{defH}). Note that for all $n\ge 1$ the set $H_i$ is a Zariski open subset of $\tilde{H}_i$. 

Observe also  that for  $n>1$ and for $i=1,2,3,4$, we have that $\mathcal H_i$,  
the set of all $E\in \mathcal M_n^{\irreg}(0)$ such that $Z_{G(E)}(i)$ is not reduced. Fix $D\in |\mathcal O_{\mathbb P^1\times \mathbb P^1}(n,1)|$. For each $i\in \{1,2,3,4\}$
the scheme $D\cap D_i$ is a zero-dimensional scheme of degree $n$. 

\begin{definition}\label{weights}
Let $m_1\ge \cdots \ge m_x>0$ be the degrees  of the connected components of $D\cap D_i$
(in decreasing order).
\begin{itemize}
\item  We say that $(m_1,\dots ,m_x)$ is the {\bf $i$-th    profile} of $D$ and of all bundles $E$ with $G(E)=D$. 
 
\item  We say that the integer $\sum (m_i-1) =m_1+\cdots +m_x-x$ is the {\bf $i$-th weight} of $D$ and of all bundles $E$ with $G(E) =D$. 
 
\item  The {\bf length} $\ell$ of the profile $D\cap D_i$ and of all  bundles $E$ with $G(E)=D$
is the maximal integer $y$ such that $m_y\ge 2$, with the convention $\ell =0$ if $m_1=1$, i.e. if $D$ is transversal to $D_i$. 

\item The {\bf reduced $i$-th profile} of $D$
is the set of integers $m_1,\dots ,m_\ell$.
\end{itemize}
\end{definition}

 If $\ell>0$, these numbers form a non-decreasing sequence of integers $\ge 2$ with $\sum _i m_i\le n$. 
 Note that the $i$-th weight is uniquely determined by the reduced profile of $m_1,\dots ,m_s$: it is $0$ if $\ell =0$, while it is $-\ell+m_1+\cdots +m_\ell$ if $\ell>0$.
For instance, for $n=2$ the possible profiles are $2$ and $1,1$; while for $n=3$ the profiles are $3$, $2,1$ and $1,1,1$.\\

\begin{notation}For any such $E$ with no jump we denote by 
$$m_1,\dots ,m_{s(i)} \quad  \text{or} \quad  m_1(E),\dots ,m_{s(i)}(E)$$
 the multiplicities of the connected components of $Z_{G(E)}(i)$ in non-decreasing order. 
By definition, the $i$th-weight $w_i(E)$ of $E\in \mathcal M_n(0)$ is the integer 
 $$w_i(E)\ce \sum _{j=1}^{s(i)} (m_j-1),$$
 and the weight $w(E)$ of $E\in \mathcal M_n(0)$ is the integer $$w(E)\ce w_1(E)+w_2(E)+w_3(E)+w_4(E).$$
\end{notation}

\begin{theorem}\label{irregw} $E\in \mathcal M_n(0)$ is irregular if and only if $w(E)>0$.
\end{theorem}
\begin{proof} This is true because $\mathcal M_n^{\irreg}(0) = \bigcup_{i=1}^4 \mathcal H_i. $
\end{proof}

For simplicity we write $h^0(\mathcal I_Z(x,y))$ instead of $h^0(\mathbb P^1\times \mathbb P^1,\mathcal I_Z(x,y))$ and the same for higher cohomology groups. 
We write $\mathcal M_{n}$ for the moduli of bundles 
with trivial determinant and $c_2=n$.

\begin{remark}\label{Woo1}
We explain here our motivation for the introduction of the profiles of an element of $|\mathcal O_{\mathbb P^1\times \mathbb P^1}(n,1)|_{sm}$ and hence of a bundle $E\in \mathcal M_n(0)$.
Since we always assume that $n\ge 2$, the graph map $G$ is surjective \cite[Th. 5.2.2]{BH}.
Hence the description of a spectral curve $C$  covering a graph $D\in |\mathcal O_{\mathbb P^1\times \mathbb P^1}(n,1)|_{sm}$ is an important  invariant for  each bundle $E$ with $G(E)=D$.
Fix $D\in  |\mathcal O_{\mathbb P^1\times \mathbb P^1}(n,1)|_{sm}$ and set $S_i:= D\cap \mathbb P^1\times \{a_i\}$, $i=1,2,3,4$. A spectral curve $C$ of a smooth graph is irreducible (Lemma \ref{Woo2}).
It is a double covering $u\colon C\to D$ ramified exactly over the points of $S:= S_1\cup S_2\cup S_3\cup S_4$. The curve $C$ is contained in the smooth surfaces
$\mathbb P^1\times T^*$ and it is smooth outside $S$. Since $p_a(C) =2n-1$, the integer $\#S$ is a measure of the singularities of $C$. Since $u\colon C\to D$ is a double covering with $D\cong \mathbb P^1$, 
$C$ is a ``hyperelliptic-like'' curve whose singular points have multiplicity $2$, but the singularity type of $C$ at $\tilde{p}$ depends on the spectral profile. Fix $p\in S$, say $p\in S_i$, and let $\tilde{p}$ be the only point of $C$ with $u(\tilde{p})=p$. Call $m_p$ the multiplicity of $p$ in $D\cap D\times \{a_i\}$. If $m_p=1$, then $C$ is smooth at $\tilde{p}$. If $m_p=2$, then $C$ has an ordinary double point at $\tilde{p}$.
\end{remark}

\begin{lemma}\label{Woo2}
Let $C\subset \mathbb P^1\times T^*$ be the spectral curve of a smooth graph $D\in|\mathcal O_{\mathbb P^1\times \mathbb P^1}(n,1)|_{sm}$. Then $D$ is irreducible.
\end{lemma}

\begin{proof}
By the definition of spectral curve there is a morphism $u\colon C\to D$ generically of degree $2$. Since $D\in|\mathcal O_{\mathbb P^1\times \mathbb P^1}(n,1)|(0)$, it has no vertical component and hence $C$ has no vertical component, i.e. the restriction of $u$ to any irreducible component is dominant. Assume that $C$ is reducible. Since $\deg(u)=2$, $C$ has exactly $2$ irreducible components, say $C=C_1\cup C_2$ with $u_{|C_i} C_i\to D$ a degree $1$ morphism between irreducible curves with a smooth target $D$. 
Hence $u_{|C_i}\colon C_i\to D$ is an isomorphism. Thus $C_i\cong \mathbb P^1$. We have $C_i\subset \mathbb P^1\times T^*$. Since $C_i\cong \mathbb P^1$ and $T^*$ is a smooth elliptic curve, the restriction to $C_i$ of the projection $\mathbb P^1\times T^*\to T^*$ is constant. Call $o_i$ its image. Thus $C_i =\mathbb P^1\times \{o_i\}$. Since $u\colon C\to D$ is generically unramified, $o_1\ne o_2$. Hence $C_1$ and $C_2$ are connected components, contradicting the fact
that $D\cap \mathbb P^1\times \{o_1\}\ne \emptyset$ and hence $u$ has at least one ramification point.
\end{proof}

\begin{notation}\label{nota}
Let $\mathcal M_n^{sm}$ denote the set of all $E\in \mathcal M_n(0)$ 
such that the spectral curve $C_E$ of $E$ is smooth,
and let $\mathcal M_n^{sg}$ be the set of all $E\in \mathcal M_n(0)$ such that $C_E$ is singular. With our notation the elements of $\mathcal M_n^{sg}$ have no jumps.
\end{notation}

It was proven in section \ref{irregVreg} that  $\mathcal M_n^{sm}= \mathcal M_n^{\reg}(0)$
and $\mathcal M_n^{sg}= \mathcal M_n^{\irreg}(0)$, and we now explore these equalities.

\begin{theorem}\label{Wrr5}
Assume $n>1$. For all $E\in \mathcal M_n^{\irreg}(0)$  we have:
\begin{enumerate}
\item $w(E)\le 2n-2$.
\item $E$ has $(n, 0,\dots, 0)$ has its $i$-th profile if and only if $w_i(E)\le n-1$ for all $i$ and $w_i(E)=n-1$ .
\item Fix $i\in \{1,2,3,4\}$, a positive integer $\ell$ and a restricted profile $(m_1,\dots ,m_\ell)$. 
There exists a non-empty locally closed analytic subset $\Gamma \in \mathcal M_n^{\irreg}(0)$
such that  all  $G(E)\in |\mathcal O_{\mathbb P^1\times \mathbb P^1}(n,1)|$ have   $(m_1,\dots ,m_\ell)$ as $i$-th profile
and this set of graphs $G(E)$  has codimension at most $-\ell +m_1+\cdots +m_\ell$ in $|\mathcal O_{\mathbb P^1\times \mathbb P^1}(n,1)|$.
\end{enumerate}
\end{theorem}

\begin{proof}
Here we use the Zariski topology of the projective space $|\mathcal O_{\mathbb P^1\times \mathbb P^1}(n,1)|$, so being general in it means ``outside finitely many proper algebraic subset'' (their union has lower dimension).

Fix a smooth $D\in |\mathcal O_{\mathbb P^1\times \mathbb P^1}(n,1)|$. Hence $D\cong \mathbb P^1$. The projection onto the second factor  $\pi_2\colon \mathbb P^1\times \mathbb P^1\to \mathbb P^1$ induces a degree $n$ morphism $f\colon  D\to \mathbb P^1$. Since $D\cong \mathbb P^1$, the Riemann--Hurwitz formula gives that, counting multiplicities, the ramification divisor $\mathcal R\subset D$ is an effective divisor of degree $2n-2$ 
(\cite[pp.\thinspace 216--219]{GH}, \cite[Cor.\thinspace  IV.2.4]{h}). Fix $p\in D$. Since $D$ is irreducible and of bidegree $(n,1)$, the scheme $Z_p:= D\cap \pi_2^{-1}(f(p))$ has degree $n$.
Note that $p\in \mathcal R_{\red}$ if and if $Z_p$ is not formed by $n$ distinct points and the multiplicity of $p$ in $\mathcal R$ is how we computed the multiplicity of the $i$-th profile of $E$. Since $w(E)$ only counts the contribution of $4$ of the fibers of $f$, we get (1) and (2).

Now we use that $G$ is surjective \cite[Th. 5.2.2]{BH}. We see that to prove item (3)  it is sufficient to prove the ``corresponding'' statement for $|\mathcal O_{\mathbb P^1\times \mathbb P^1}(n,1)|_{sm}$. Fix $\ell$ distinct points $p_1,\dots ,p_\ell$. Let $Z\subset D_i$ be the connected zero-dimensional subscheme of $D_i$ with $Z_{\red} =\{p_1,\dots ,p_s\}$ and the connected component of $Z$ with $p_h$ as its reduction has degree $m_h$.
Set $m:= m_1+\cdots +m_\ell$. \\

{ (a)} In this step we prove that $h^1(\mathbb P^1\times \mathbb P^1,\mathcal I_Z(n,1)) =0$; hence $h^0(\mathcal I_Z(n,1)) =h^0(\mathcal O_{\mathbb P^1\times \mathbb P^1}(n,1)) -m$.
Since $Z\subset D_i$, we have an exact sequence
\begin{equation}
0\to \mathcal O_{\mathbb P^1\times \mathbb P^1}(n-1,1)\to \mathcal I_Z(n,1)\to \mathcal I_{Z,D_i}(n,1)\to 0
\end{equation}
The K\"{u}nneth formula give $h^1(\mathcal O_{\mathbb P^1\times \mathbb P^1}(n-1,1))=0$. Since $\deg(Z)\le n$, $D_i\cong \mathbb P^1$ and $\deg(\mathcal O_{D_i}(n,1)) =n$, we have $h^1(D_i,\mathcal I_{Z,D_i}(n,1)) =0$.  \\

 Take a general $D\in |\mathcal I_Z(n,1)|$. In this step we prove that $D$ is smooth,  transversal to each $D_j$, $j\ne i$, and that $(m_1,\dots ,m_\ell)$ is the restricted $i$-th profile of $D$. Let $\mathcal B$ denote the set-theoretic base locus of $|\mathcal I_Z(n,1)|$. By the theorem of Bertini to prove that $D$ is smooth outside $\{p_1,\dots ,p_\ell\}$ it is sufficient to prove that
$\{p_1,\dots ,p_\ell\}=\mathcal B$ (\cite[p. 137]{GH}, \cite[III.10.9]{h}, \cite[6.3]{jou}), i.e. that $h^0(\mathcal I_{Z\cup \{q\}}(n,1)) =h^0(\mathcal I_Z(n,1)) -1$ for all $q\in \mathbb P^1\times \mathbb P^1\setminus \{p_1,\dots ,p_\ell\}$.\\

{(b1)} Take $q\in \mathbb P^1\times \mathbb P^1\setminus D_i$. Since $q\notin D_i$ and $Z\subset D_i$, we have the exact sequence 
\begin{equation}\label{eqa2}
0\to \mathcal I_q(n-1,1)\to \mathcal I_{Z\cup \{q\}}(n,1)\to \mathcal I_{Z,D_i}(n,1)\to 0
\end{equation}
We saw that $h^1(D_i,\mathcal I_{Z,D_i}(n,1))=0$. Since $n\ge 0$, $\mathcal O_{\mathbb P^1\times \mathbb P^1}(n-1,1)$ is globally generated and hence $h^1(\mathcal I_q(n-1,1)) =0$. Thus, the 
long cohomology exact sequence of \eqref{eqa2} gives $h^1(\mathcal I_{Z\cup \{q\}}(n,1))=0$, proving that $q\notin \mathcal B$.\\

{(b2)} Take $q\in D_i\setminus \{p_1,\dots ,p_\ell$. Since $Z\cup \{q\}\subset D_i$, we have an exact sequence
\begin{equation}\label{eqa3}
0\to \mathcal O_{\mathbb P^1\times \mathbb P^1}(n-1,1)\to \mathcal I_{Z\cup \{q\}}(n,1)\to \mathcal I_{Z\cup \{q\},D_i}(n,1)\to 0
\end{equation}
Since $\deg(Z\cup \{q\})=\deg(Z)+1\le n+1$, $D_i\cong \mathbb P^1$ and $\deg(\mathcal O_{D_i}(n,1)) =n$, we have $h^1(D_i,\mathcal I_{Z\cup \{q\},D_i}(n,1)) =0$. Hence $q\notin \mathcal B$.\\

{(b3)} By steps(b1) and(b2), $D$ is smooth outside the finite set $\{p_1,\dots ,p_\ell\}$. Since we are using the Zariski topology,
any finite intersection of non-empty Zariski open subsets of the projective space $|\mathcal O_{\mathbb P^1\times \mathbb P^1}(n,1)|$ is non-empty and hence Zariski dense in $|\mathcal O_{\mathbb P^1\times \mathbb P^1}(n,1)|$. Since being smooth is an open condition, it is sufficient to find $A\in |\mathcal I_Z(n,1)|$ smooth at all points of $\{p_1,\dots ,p_\ell\}$.
Take a general $B\in |\mathcal O_{\mathbb P^1\times \mathbb P^1}(n-1,1)|$ and set $A:= D_i\cup B$. Since $Z\subset D_i$, $Z\subset A$. Since $B$ is general and $\mathcal O_{\mathbb P^1\times \mathbb P^1}(n-1,1)$ is globally generated, $B\cap \{p_1,\dots ,p_\ell\}=\emptyset$. Hence  $A$ is smooth at all points of $\{p_1,\dots ,p_\ell\}$.\\

{(b4)} In this step we prove that $D$ is transversal to each $D_j$, $j\ne i$. Since we are using the Zariski topology in which finite intersections of non-empty subsets of $|\mathcal I_Z(n,1)|$ are non-empty, it is sufficient to find $A\in |\mathcal I_Z(n,1)|$ which is transversal to each $D_j$, $j\ne i$. Take as $A$ the union $D_i$ and $n$ distinct elements $R_1,\dots ,R_n$ of $|\mathcal O_{\mathbb P^1\times \mathbb P^1}(1,0)|$. We have $D_i\cap D_j=\emptyset$ and each $R_h$ is transversal to all elements
of $|\mathcal O_{\mathbb P^1\times \mathbb P^1}(1,0)|$.\\

{ (b5)} In this step we prove that $m_1,\dots ,m_\ell$ is the restricted profile of $D$. We saw in step(b2) that $\mathcal B\cap D_i=\{p_1,\dots,p_\ell\}$.
Hence $|\mathcal I_Z(n,1)|$ induces a base point free linear system $\mathcal A$ on $D_i$. By the theorem of Bertini its general element is formed by distinct points outside
$\{p_1,\dots ,p_\ell\}$ (\cite[p.\thinspace 137]{GH}, \cite[III.10.9]{h}, \cite[6.3]{jou}). Hence $D$ is transversal outside $\{p_1,\dots ,p_\ell\}$, i.e. the restricted $i$-th profile of $D$ has exactly $\ell$ entries. Since $\{p_1,\dots ,p_\ell\}$ is a finite set and we are using the Zariski topology, it is sufficient to prove that for each $h\in \{1,\dots ,\ell\}$ a general $D$ has the property that $D\cap D_i$ has multiplicity $m_h$ at $p_h$. Hence it is sufficient to prove that it has multiplicity $\le m_h$ at $p_h$, given that by the definition of $Z$ every smooth element $D\in |\mathcal{I}_Z(n,1)|$ has intersection with $D_i$ of multiplicity at least $m_h$ at $p_h$.
Set $Z'= Z\cup Z((m_h+1)p)$.
Since $Z(m_hp_h)\subseteq Z$, we have $\deg(Z') =\deg(Z)+1$. As in step {(b2)} we see that $h^1(\mathcal I_{Z'}(n,1)) =0$. Hence having contact with multiplicity $>m_h$ only occurs in codimension one in $|\mathcal I_Z(n,1)|$. Hence the general $D$ has $m_1,\dots ,m_\ell$ as its restricted $i$-th profile.\\

{(c)} By step {(b)} we get that the set of all $D\in |\mathcal O_{\mathbb P^1\times \mathbb P^1}(n,1)|_{sm}$ containing $Z$ is a nonempty subset of codimension $m$ and  that a general element $D$ of it has the property that $w_j(D)=0$ for all $j\ne i$ and $m_1,\dots ,m_\ell$ is the restricted $i$-profile of $D$. Note that $2$ elements $D, D'\in |\mathcal O_{\mathbb P^1\times \mathbb P^1}(n,1)|_{sm}$ such that $D\cap D_i\ne D'\cap D_i$ are distinct.
Hence taking the union for all choices of $s$ distinct points of $D_i$ we get item (3).
\end{proof}

\section{Ramifications and   weights}

\begin{proposition}\label{Wrr6}
Fix an integer $n\ge 2$. Take a general $D\in |\mathcal O_{\mathbb P^1\times \mathbb P^1}(n,1)|$. Then $\pi_{2|D}: D\to \mathbb P^1$ has $2n-2$ distinct ramification points and their images are $2n-2$ distinct points of $\mathbb P^1$.
\end{proposition}

\begin{proof}
Fix $p\in \mathbb P^1\times \mathbb P^1$ and let $J$ be the only element of  $|\mathcal O_{\mathbb P^1\times \mathbb P^1}(1,0)|$ containing $p$. Let $Z\subset J$ be the degree $3$ zero-dimensional subscheme of $J$ with degree $3$. Since $n\ge 2$,  step(b2) of the proof of Theorem \ref{Wrr5} with $J$ instead of $D_i$ gives $h^1(\mathcal I_Z(n,1)) =0$. Hence the set of all $D\in |\mathcal O_{\mathbb P^1\times \mathbb P^1}(n,1)|$ with a non-ordinary ramification point at $p$ has codimension $3$ in $|\mathcal O_{\mathbb P^1\times \mathbb P^1}(n,1)|$. Since $\dim \mathbb P^1\times \mathbb P^1 =2$, we get that a general $D\in |\mathcal O_{\mathbb P^1\times \mathbb P^1}(n,1)|$  has only ordinary ramification points, i.e., it has $2n-2$ distinct ramification points. If $2\le n\le 3$ the theorem of Bezout shows that no irreducible $D\in |\mathcal O_{\mathbb P^1\times \mathbb P^1}(n,1)|$ may have two distinct ramification points contained in the same element of  $|\mathcal O_{\mathbb P^1\times \mathbb P^1}(1,0)|$. Now assume $n\ge 4$. Fix $q\in J$ such that $q\ne p$ and let $A\subset J$ be the degree $4$ zero-dimensional scheme with $\{p,q\}$ as its reduction and both connected component of $A$ of degree $2$. Since $n\ge 3$,
the proof of step(b2) of Theorem \ref{Wrr7} below gives that $h^1(\mathcal I_A(n,1)) =0$, i.e. the set of all  $D\in |\mathcal O_{\mathbb P^1\times \mathbb P^1}(n,1)|$ with both $p$ and $q$ as some of their ramification points has codimension $4$ in  $|\mathcal O_{\mathbb P^1\times \mathbb P^1}(n,1)|$. Since $\dim|\mathcal O_{\mathbb P^1\times \mathbb P^1}(1,0)|=1$ and for each $J\in \mathcal O_{\mathbb P^1\times \mathbb P^1}(1,0)|$ the set of all subsets of $J$ with cardinality $2$ has dimension $2$, a general $D\in |\mathcal O_{\mathbb P^1\times \mathbb P^1}(n,1)|$ has $2n-2$
distinct ramification points with $2n-2$ distinct images in $\mathbb P^1$.
\end{proof}

\begin{theorem}\label{Wrr7}
Fix an ordering $i_1,i_2,i_3,i_4$ of the set $\{1,2,3,4\}$.
\begin{itemize}
\item[(a)] For all $n\ge 2$ there exists  a codimension $2$ locally closed analytic subset $\Gamma \subset \mathcal M_n(0)$ such that $w_{i_1}(E)=w_{i_2}(E)=1$ and $w_{i_3}(E)=w_{i_4}(E)=0$ for all $E\in \Gamma$.

\item[(b)] For all $n\ge 4$ there exists  a codimension $3$ locally closed analytic subset $\Delta \subset \mathcal M_n(0)$ such that $w_{i_1}(E)=w_{i_2}(E)=w_{i_3}(E)=1$ and $w_{i_4}(E)=0$ for all $E\in \Delta$.
\end{itemize}
\end{theorem}

\begin{proof}
To simplify the notation we use $j$ instead of $i_j$. The other permutations only need more double or triple indices.

We always work in $|\mathcal O_{\mathbb P^1\times \mathbb P^1}(n,1)|_{sm}$. The surjectivity of the graph map $G$ \cite[Th. 5.2.2]{BH} gives that it sufficient to find
$\Gamma_1\subset |\mathcal O_{\mathbb P^1\times \mathbb P^1}(n,1)|$, $n\ge 2$, (resp. $\Delta_1\subset |\mathcal O_{\mathbb P^1\times \mathbb P^1}(n,1)|$, $n\ge 4$) such that
$\dim \Gamma _1 =\dim |\mathcal O_{\mathbb P^1\times \mathbb P^1}(n,1)| -2$ (resp. $\dim \Delta _1 =\dim |\mathcal O_{\mathbb P^1\times \mathbb P^1}(n,1)| -3$) such that all $C\in \Gamma_1$ (resp. $C\in \Delta_1$) have this
 property.
By Lemma \ref{Wrr6} there is $A\in |\mathcal O_{\mathbb P^1\times \mathbb P^1}(n,1)|_{sm}$ such that $\pi_{2|A}$ has $2n-2$ ramification points with $2n-2$ different images $a_1,\dots ,a_{2n-2}\in \mathbb P^1$. The group $\mathrm{Aut}(\mathbb P^1)$ acts uniquely 3-transitively on $\mathbb P^1$, i.e. for any $2$ triples $(b_1,b_2,b_3)$, $(c_1,c_2,c_3)$ of distinct points of $\mathbb P^1$ there is a unique $\gamma\in \mathrm{Aut}(\mathbb P^1)$ such that $\gamma(b_j)=c_j$ for all $j$.\\

 (a) In this step we prove part (a).  We saw that
there is $\alpha_2 \in \mathrm{Aut}(\mathbb P^1)$ such that $\alpha_2(a_1) =\pi_2(D_1)$ and $\alpha_2(a_2)= \pi_2(D_2)$. Let $\alpha$ denote the automorphism of $\mathbb P^1\times \mathbb P^1$ which acts as the identity 
on the first factor and as $\alpha_2$ on the second factor. For all $E\in \mathcal M_n$ such that $G(E)=\alpha(A)$
we have $w_1(E)=w_2(E)=1$, because all ramification points of $D:= \alpha(A)$ are ordinary, 
exactly one of  them is contained in $D_1$ and exactly one of them is contained in $D_2$. 
If $n=2$ we also have $w_3(E)=w_4(E)=0$, because $w(E)\le 2$ for all $E\in \mathcal M_{2}(0)$ 
(Theorem \ref{Wrr5}). However, for $n>2$ to get $w_3(E)=w_4(E)=0$ we need to use some dimensional count.
Let $p$ (resp. $q$) denote the ramification point of $D$ contained in $D_1$ (resp. $D_2$). Let $(2p,D_1)$ (resp. $(2q,D_2)$) denote the degree $2$ zero-dimensional subscheme of $D_1$ (resp. $D_2$) with $p$ (resp. $q$) as its reduction. Set $Z:= (2p,D_1)\cup (2q,D_2)$. Since $p$ and $q$ are ramification points of $D$, $Z$ is a degree $4$ subscheme of $D$. Hence there is an exact sequence of sheaves
\begin{equation}\label{eqb1}
0\to \mathcal O_{\mathbb P^1\times \mathbb P^1}\to \mathcal I_Z(n,1)\to \mathcal I_{Z,D}(n,1)\to 0
\end{equation}
 The K\"{u}nneth formula gives $h^1(\mathcal O_{\mathbb P^1\times \mathbb P^1})=0$. Since $\deg(Z)=4$, $D\cong \mathbb P^1$, $n\ge 2$ and $\deg(\mathcal O_D(n,1)) =2n$, we have
$h^1(D,\mathcal I_{Z,D}(n,1)) =0$, i.e. $\dim |\mathcal I_{Z,D}(n,1)| =\dim |\mathcal O_{\mathbb P^1\times \mathbb P^1}(n,1)|-4$. 

Take $p'\in D_1$ and $q'\in D_2$ and set
$Z(p',q'):= (2p',D_1)\cup (2q',D_2)$. By the semicontinuity theorem for cohomology  \cite[Th. III.13.8]{h} we have  $h^1(\mathcal I_{Z(p',q')}(n,1))=0$
for all $(p',q')$ in a Zariski neighborhood of $(p,q)$ in $D_1\times D_2$. Hence there is an irreducible locally closed and codimension $2$ algebraic subset $\Gamma_1$ of $|\mathcal O_{\mathbb P^1\times \mathbb P^1}(n,1)|_{sm}$ such that $D\in \Gamma_1$ and
all $X\in \Gamma_1 $  have an ordinary ramification point contained in $D_1$ and an ordinary ramification point contained in $D_2$. 

Since $D$ has only ordinary ramification points and $p$ and $q$ are its only ramification points contained in $D_1\cup D_2$, there is a Zariski open neighborhood $\Gamma_2$ of $D$ in $\Gamma_1$ such that all $C\in \Gamma_1$ have only ordinary ramification points, $D_1$ contains only one ramification point of $C$ and $D_2$
contains only one ramification point of $C$.

For $j=3,4$ let $\Gamma(j)$ denote the set of all $C\in \Gamma_2$ which are transversal to $D_j$. Since transversality is an open condition for the Zariski topology, $\Gamma(j)$ is Zariski open in $\Gamma_2$. Note that any bundle $E$ with $G(E)\in \Gamma(3)\cap \Gamma(4)$ satisfies $w_1(E)=w_2(E)=1$ and $w_3(E)=w_4(E)=0$. Assume for the moment $\Gamma(3)\ne \emptyset$ and $\Gamma(4)\ne \emptyset$. Since $\Gamma_2$ is irreducible of dimension $\dim |\mathcal O_{\mathbb P^1\times \mathbb P^1}(n,1)|-2$, the same is true for $\Gamma(3)\cap \Gamma(4)$. Hence part (a) holds if $\Gamma(3)\ne \emptyset$ and $\Gamma(4)\ne \emptyset$.
Assume for instance $\Gamma(3)=\emptyset$, i.e. assume that all $C\in \Gamma_2$ are tangent to $D_3$. In particular this holds for $D$. Call $a$ the point of $D_3\cap D$ at which $D_3$ and $D$ are tangent. Let $(2a,D_3)$ denote the degree $2$ zero-dimensional subscheme of $D_3$ with $a$ as its reduction. Set $W:= Z\cup (2a,D_3)$. Since $p$, $q$ and $a$ are ramification points of $D$, $W\subset D$. Hence there is an exact sequence of sheaves
\begin{equation}\label{eqb2}
0\to \mathcal O_{\mathbb P^1\times \mathbb P^1}\to \mathcal I_W(n,1)\to \mathcal I_{W,D}(n,1)\to 0
\end{equation}
 The K\"{u}nneth formula gives $h^1(\mathcal O_{\mathbb P^1\times \mathbb P^1})=0$. Since $\deg(Z)=4$, $D\cong \mathbb P^1$, $n\ge 3$ and $\deg(\mathcal O_D(n,1)) =2n$, we have
$h^1(D,\mathcal I_{W,D}(n,1)) =0$, i.e. $\dim |\mathcal I_{W,D}(n,1)| =\dim |\mathcal O_{\mathbb P^1\times \mathbb P^1}(n,1)|-6$. Since $\dim (D_1\times D_2\times D_3)=3$, varying the points $(p',q',a')
\in D_1\times D_2\times D_3$, we get that $\Gamma_2\setminus \Gamma(3)$ has codimension at least $1$ in $\Gamma_2$.
Hence $\Gamma(3)\ne \emptyset$, contradicting one of our assumptions. \\

 (b) In this step we prove part (b). Let $\gamma_2$ be the only element of  $\mathrm{Aut}(\mathbb P^1)$ such that $\gamma_2(a_1)=\pi_2(D_1)$, $\gamma_2(a_2)= \pi_2(D_2)$ and $\gamma_2(a_3)=\pi_2(D_3)$. Let $\gamma$ be the automorphism of $\mathbb P^1\times \mathbb P^1$ which acts as the identity on the first factor and as $\gamma_2$ on the second factor. Set $X:= \gamma(A)$. 
 
 Let $p$ (resp. $q$, resp $a$) denote the contact locus of $X$ and $D_1$ (resp. $D_2$, resp. $D_3$). Note that $w_1(E)=w_2(E)=w_3(E)=1$ for all bundles $E$ such that $G(E)=X$. Hence to prove part (b) it is sufficient to count the curves $X'$ near $X$ and tangent to $D_j$ for $j=1,2,3$ and prove that the general such curve is not tangent to $D_4$. Let $p$ (resp. $q$, resp. $a$) be the point of contact of $X$ and $D_1$ (resp. $D_2$, resp. $D_3$). As in step (a) let $(2p,D_1)$ (resp. $(2q,D_2)$, resp. $(2a,D_3)$) denote the degree $2$ connected zero-dimensional subscheme of $D_1$ (resp. $D_2$, resp. $D_3$) with $p$ (resp. $q$, resp. $a$) as its reduction. Set 
 $$W:= (2p,D_1)\cup (2q,D_2)\cup(2a,D_3).$$ 
 Since $X$ ramifies at $p$, $q$ and $a$, $W\subset X$. We get an exact sequence similar to \eqref{eqb2} with $X$ instead of $D$.
As in step (a) we get $h^1(\mathcal I_W(n,1)) =0$, i.e. $\dim |\mathcal I_W(n,1)| =\dim |\mathcal O_{\mathbb P^1\times \mathbb P^1}(n,1)|-6$. For any$(p',q',a')\in D_1\times D_2\times D_3$ and set 
$$W(p',q',a'):= (2p',D_1)\cup (2q',D_2)\cup(2a',D_3).$$ By the semicontinuity theorem for cohomology  \cite[Th. III.13.8]{h} we have  $h^1(\mathcal I_{W(p',q',a')}(n,1)) =0$, i.e. $\dim |\mathcal I_{W(p',q',a')}(n,1)| =\dim |\mathcal O_{\mathbb P^1\times \mathbb P^1}(n,1)|-6$ for all $(p',q',a')$ in a non-empty open subset of $(p,q,a)$ in  $D_1\times D_2\times D_3$. Since $\dim (D_1\times D_2\times D_3)=3$, we get
a codimension $3$ analytic subset $\Delta'$ of $\mathcal M_n(0)$ such $w_1(E)=w_2(E)=w_3(E)=1$. 

To prove step (b) we need to prove that, outside a lower dimensional analytic subset of $\Delta'$, the bundles satisfy $w_4(E)=0$, i.e. their graph is transversal to $D_4$. As in step (a) we have a locally closed irreducible subset $\Delta_2$ of $|\mathcal O_{\mathbb P^1\times \mathbb P^1}(n,1)|(0)$ with codimension $3$, containing $X$ and formed by smooth curves with ordinary ramifications, $2n-2$ images of the ramification points. We need to prove that a general $X'\in \Delta_2$ is transversal to $D_4$. Assume that this is not true. In particular $X$ is tangent to $D_4$. Since the ramification points of $X$ have different, there is a unique $b\in D_4\cap X$ at which $X$ and $D_4$ are tangent.
Let $(2b,D_4)$ denote the degree $2$ zero-dimensional subscheme of $D$ with $b$ as its reduction. Set $W':= W\cup (2b,D_4)$.
Note that $\deg(W')=8$ and $W'\subset X$. We get an exact sequence of sheaves similar to \eqref{eqb2} with $X$ instead of $D$ and $W'$ instead of $X$.
Since $n\ge 4$, $X\cong \mathbb P^1$, $\deg(W')=8$ and $\deg(\mathcal O_X(n,1)) =2n$, we have $h^1(X,\mathcal I_{W',X}(n,1)) =0$ and hence
$h^0(\mathcal I_{W'}(n,1)) =h^0(\mathcal I_W(n,1)-2$. Since $\dim D_4 =1$, repeating the last part of step (a) we get a contradiction.
\end{proof}

\begin{example}\label{Wrr8}
Take $n=2$. Let $E$ be any irregular bundle. By Theorem \ref{Wrr5} we have $1\le w(E)\le 2$ and $w_i(E) \le 1$ for all $i$. Theorems \ref{Wrr5} and \ref{Wrr7} show
that for all possible quadruples $(x_1,x_2,x_3,x_4)$ of integers with $0\le x_i\le 1$ and $1\le x_1+x_2+x_3+x_4\le 2$ there is $E\in \mathcal M_{\delta,2}$ such that $w_i(E) =x_i$ for all $i$. We also got that the part corresponding to $(x_1,x_2,x_3,x_4)$ has dimension at least $8-x_1-x_2-x_3-x_4$.
\end{example}

\begin{theorem}\label{rr9}
Fix an integer $n\ge 2$. Fix an ordering $i_1,i_2,i_3,i_4$ of $\{1,2,3,4\}$. Then there exists an irreducible locally closed analytic subset of dimension $2n+2$ of $\mathcal M_n(0)$
such that for all $E\in \Gamma$ the restricted $i_1$-th and $i_2$-th profiles of $E$ are $(n)$, while $w_{i_3}(E)=w_{i_4}(E)=0$.
\end{theorem}

\begin{proof}
Since the general case is similar, we only do the case $i_j=j$ for $j=1,2,3,4$.
Note that if $E\in \mathcal M _{\delta,n}(0)$ has restricted profile $(n)$ for $D_1$ and $D_2$, then $w_1(E)=w_2(E)=n-1$ and hence $w_3(E)=w_4(E)=0$ by Theorem \ref{Wrr5}. Hence we do not have the most difficult part of the proof of Theorems \ref{Wrr5} and \ref{Wrr7}, transversality with respect  to some $D_j$.

Since the case $n=2$ is covered by Example \ref{Wrr8}, we assume $n\ge 3$.

Since the graph map is surjective and all its fibers have dimension at least $2n-1$, it is sufficient to find a $3$-dimensional complex analytic family $\Gamma_1\subset  |\mathcal O_{\mathbb P^1\times \mathbb P^1}(n,1)|$ such that each $D\in \Gamma_1$ has profile $(n)$ with respect to $D_1$ and $D_2$.

Fix $p=(p_1,p_2)\in D_1$ and $q=(q_1,q_2)\in D_2$ such that $p_1\ne q_1$. For any integer $x>0$ let $(xp,D_1)$ (resp. $(xq,D_2)$) denote the connected and degree $x$ zero-dimensional subcheme
of $D_1$ (resp. $D_2$) with $p$ (resp. $q$) as its reduction. Set $Z:= (np,D_1)\cup (nq,D_2)$. We have $\deg(Z) =2n$ and hence $\dim |\mathcal I_Z(n,1)| \ge \dim |\mathcal O_{\mathbb P^1\times \mathbb P^1}(n,1)|-2n=1$.
Take a general $D\in |\mathcal I_Z(n,1)|$. Since the set of all such pairs $(p,q)\in D_1\times D_2$ has dimension $2$ and smoothness is an open condition for the Zariski topology, to conclude the proof of the theorem it is sufficient to prove that $D$ is smooth.

 Fix $a\in D_1\setminus \{p\}$ and $b\in D_2\setminus \{q\}$. Let $R_1$ be the only element of $|\mathcal O_{\mathbb P^1\times \mathbb P^1}|$ containing $p$
 and $R_2$ the only element of $|\mathcal O_{\mathbb P^1\times \mathbb P^1}|$ containing $q$. Since $p_1\ne q_1$, we have $R_1\ne R_2$, $q\notin R_1$ and $p\notin R_2$. Set $U:= \mathbb P^1\times \mathbb P^1\setminus (D_1\cup D_2\cup R_1\cup R_2)$. Set $u:= D_1\cap R_2$ and $v:= D_2\cap R_1$.\\

{(a)} In this step we prove that $h^0(\mathcal I_Z(n,1))=2$, i.e. that $\dim |\mathcal I_Z(n,1)|=1$.
To prove that $h^0(\mathcal I_Z(n,1))=2$ it is sufficient to prove that $h^0(\mathcal I_{Z\cup \{a,b\}}(n,1))=0$. Assume $h^0(\mathcal I_{Z\cup \{a,b\}}(n,1))>0$ and take $X\in |\mathcal I_{Z\cup \{a,b\}}(n,1)|$. 

Since $\{a\}\cup (np,D_1)$ and $\{b\}\cup (nq,D_2)$ have degree $n+1$ and since $D_1$ and $D_2$ are irreducible, the theorem of Bezout gives
$D_1\cup D_2\subseteq X$. Since $h^0(\mathcal O_{\mathbb P^1\times \mathbb P^1}(n,1)(-D_1-D_2)) =h^0(\mathcal O_{\mathbb P^1\times \mathbb P^1}(n,-1))=0$, we get a contradiction.\\

 {(b)} By the theorem of Bertini  (\cite[p. 137]{GH}, \cite[III.10.9]{h}, \cite[6.3]{jou}) to prove that a general $D\in |\mathcal I_Z(n,1)|$ is smooth outside $\{p,q\}$ it is sufficient to prove that $\{p,q\}$ is the set-theoretic base locus $\mathcal B$ of $|\mathcal I_Z(n,1)|$. Recall that $p_1\ne q_1$ by assumption. Since $D_1\cap  D_2 =\emptyset$, we have $p_2\ne q_2$. Recall that $\mathrm{Aut}(\mathbb P^1)$ is simply $3$-transitive, i.e. for all triples
of $(b_1,b_2,b_3)$ and $(c_1,c_2,c_3)$ of distinct points of $\mathbb P^1$ there is a unique $h\in \mathrm{Aut}(\mathbb P^1)$ such that $h(b_i)=c_i$ for all $i=1,2,3$.
Hence the subset $\mathbb G$ of all $u\in \mathrm{Aut}(\mathbb P^1)\times \mathrm{Aut}(\mathbb P^1)$ such that $u(p)=p$ and $u(q) =q$, is an algebraic linear group isomorphic
to $(\mathbb C^\ast)^2$ and whose actions on $\mathbb P^1\times \mathbb P^1$ has the following $9$ orbits: 
$$U, \quad  D_1\setminus \{p,u\}, \quad  D_2\setminus \{q,v\}, \quad R_1\setminus \{p,v\},  
\quad R_2\setminus \{q,u\},$$ and the $4$ singletons $u, \,\, v, \,\,  p, $ and $q$. 

Assume $\mathcal B\ne \{p,q\}$, take $z\in \mathcal B\setminus \{p,q\}$ and let $U_z$ be the orbit of $\mathbb G$ containing $z$. Since each $h\in \mathbb G$ fixes scheme-theoretically $z$, we have $U_z\subset \mathcal B$.
If $U_z=U$, then we get $h^0(\mathcal I_Z(n,1)) =0$, a contradiction. \\

Now assume $U_z=u$, i.e. $z=u$. Hence $h^0(\mathcal I_{Z\cup \{u\}}(n,1)) =h^0(\mathcal I_Z(n,1)) =2$. Since $\deg(D_1\cap (Z\cup \{u\})
=n+1$, the theorem of Bezout gives $D_1\subset \mathcal B$. Since $Z\cup D_1 =(nq,D_2)$, $\mathcal I_{D_1}(n,1)  \cong \mathcal O_{\mathbb P^1\times \mathbb P^1}(n,0)$ and $q\notin D_1$, we get $h^0(\mathcal I_Z(n,1)) =h^0(\mathcal I_{(nq,D_2)}(n,0))$. Since the scheme $\pi_1((nq,D_1))$ is the degree $n$ zero-dimensional subscheme of $\mathbb P^1$ with $q_1$ as its reduction, we have $h^0(\mathcal I_{(nq,D_2)}(n,0))=h^0(\mathbb P^1,\mathcal I_{\pi_1((nq,D_1)}(n,0)) =h^0(\mathbb P^1,\mathcal O_{\mathbb P^1})=1$. 
Hence $h^0(\mathcal I_Z(n,1)) =1$, a contradiction. 
Similarly, we exclude the case $z=v$.\\

Now assume $U_z =D_1\setminus \{p,u\}$. Since $\mathcal B$ is closed in $\mathbb P^1\times \mathbb P^1$,
$D_1\subset \mathcal B$. Hence $u\in \mathcal B$. We excluded this case. Similarly, we exclude the case $U_z = D_2\setminus \{q,v\}$.\\

Now assume $U_z =R_1\setminus \{p,v\}$. Since $\mathcal B$ is closed in $\mathbb P^1\times \mathbb P^1$,
$R_1\subset \mathcal B$. Hence $v\in \mathcal B$. We excluded this case.
Similarly,  we exclude the case $U_z=R_2\setminus \{q,u\}$.\\

{(c)} In this step we conclude the proof of the theorem proving that $D$ is smooth at $p$ and at $q$. Since $|\mathcal I_Z(n,1)|$ is an irreducible variety, $\{p,q\}$ is a finite set and smoothness is an open condition for the Zariski topology, it is sufficient to prove that  a general $D'\in |\mathcal I_Z(n,1)|$ is smooth at $p$ and a general $D''\in |\mathcal I_Z(n,1)|$ is smooth at $q$. We only prove that  a general $D'\in |\mathcal I_Z(n,1)|$ is smooth at $p$, since the smoothness at $q$ only requires notational modifications. Assume that a general $D'\in |\mathcal I_Z(n,1)|$ is singular at $p$. Let $(2p,R_1)$ denote the degree $2$ zero-dimensional subscheme of $R_1$ with $p$ as its reduction. Set $W\ce Z\cup (2p,R_1)$. Since a general $D'\in |\mathcal I_Z(n,1)|$ is singular at $p$, $W\subset D'$ and hence $h^0(\mathcal I_W(n,1)) =h^0(\mathcal I_Z(n,1)) =2$. Since $D'$ contains $W$, it contains the degree $2$ subscheme $(2p,R_1)$. Since  $\mathcal O_{R_1}(n,1)$ is the degree $1$ line bundle on $R_1$,
the base locus $\mathcal B$ of $|\mathcal I_Z(n,1)|$ contains $R_1$. We excluded this case in step {(b)}.

\end{proof}

\section{The topology of $\mathcal M_n^{\reg}(0)$}

Fix an integer $n\ge 2$. Recall that $\dim |\mathcal O_{\mathbb P^1\times \mathbb P^1}(n,1)|=2n+1$. Let $\mathcal C_n\subset |\mathcal O_{\mathbb P^1\times \mathbb P^1}(n,1)|$ denote the set of all singular $D\in |\mathcal O_{\mathbb P^1\times \mathbb P^1}(n,1)|$. As usual in Complex Analysis and in Algebraic Geometry, $\Delta$ is a hypersurface of $ |\mathcal O_{\mathbb P^1\times \mathbb P^1}(n,1)|$.
 It is easy to check that $\mathcal C_n$ contains a codimension one subset formed by all $A\cup B$ with $A\in |\mathcal O_{\mathbb P^1}(1,0)|$ (which has dimension $1$) and $B$ and element of $|\mathcal O_{\mathbb P^1\times \mathbb P^1}(n-1,1)|$. The set $\mathcal C_n$ is irreducible \cite[Ch.\thinspace 1]{GKZ}
 and hence it is given by a unique equation. 
Therefore,  the set $\mathcal B_n:=  |\mathcal O_{\mathbb P^1\times \mathbb P^1}(n,1)|\setminus \mathcal C_n$ is a smooth and connected affine variety of complex dimension $2n+1$.

 By \cite{af,hl} we obtain that $\mathcal B_n$ is homotopy equivalent to a finite CW complex of real dimension at most $\dim _{\mathcal C} \mathcal B_n =2n+1$.
Let $H_1\cup H_2\cup H_3\cup H_4$ be the set of smooth, but irregular graphs, i.e. (since $n\ge 2$) the images of $\mathcal M_n^{\irreg}(0)$. Hence 
$$\mathcal A_n:= \mathcal B_n\setminus (H_1\cup H_2\cup H_3\cup H_4) =G(\mathcal M_n^{\reg}(0))$$ and using 
\cite{af,hl} we get that it 
is a smooth affine variety with the homotopy type of a finite CW complex of real dimension at most $2n+1$.

The graph map $G\colon \mathcal M_n^{\reg}(0)\to \mathcal A_n$ is a smooth submersion whose fibers are compact differential manifolds diffeomorphic to $(S^1)^{4n-2}$, where $S^1$ is the unit circle.
It then follows that the cohomology ring $H^*(\mathcal M_n^{\reg},\mathbb Z)$ is the tensor product of the cohomology ring $H^*(\mathcal A_n,\mathbb Z)$ and the cohomology ring $H^*((S^1)^{4n-2},\mathbb Z)$. We restate this fact as a theorem.

\begin{theorem}\label{full1bis}
Fix an integer $n\ge 2$. We have:
\begin{itemize}
\item[(i)]  \quad $H^*(\mathcal M_n^{\reg}(0),\mathbb Z)= H^*(\mathcal A_n,\mathbb Z)\otimes H^*((S^1)^{4n-2},\mathbb Z),$
and 
\item[(ii)] \quad  $H^i(\mathcal M_n^{\reg}(0),\mathbb Z)=0$ \ for all $i\ge 6n$.
\end{itemize}
\end{theorem}

\begin{proof}
The graph map $G\colon \mathcal M_n^{\reg}(0)\to \mathcal A_n$ is a $C^\infty$-submersion with compact fibers. The fibers of $G$ are biholomorphic to compact complex $(2n-1)$-tori and hence they are diffeomorphic to $(S^1)^{4n-2}$.
 Since $G$ is a $C^\infty$-submersion with compact fibers it is locally trivial on the base $\mathcal A_n$, i.e. for each $p\in \mathcal A_n$ there is an Euclidean open subset $U$ such that $G^{-1}(U)$ is fiberwise diffeomorphic to $U\times (S^1)^{4n-2}$ , i.e. it commutes with the projection $U\times (S^1)^{4n-2}\to U$. Since being a Serre fibration is a local condition on the base and their products are Serre fibrations,
$G$ is a Serre fibration. Thus, part (a) follows from a theorem of Leray and Hirsch \cite[17.8.1]{tdie}. 

Since $\mathcal A_n(\reg)$ is an affine variety, it has the homotopy type of a finite CW-complex of (real!) 
dimension at most $\dim  \mathcal A_n = 2n+1$ \cite{af,hl}, $H^i(\mathcal A_n,\mathbb Z)=0$ for all  $i\ge 2n+2$. Hence part (ii) follows from part (i) (alternatively, one could use the Leray--Serre spectral sequence of $G$ of the cohomology with $\mathbb Z$-coefficient \cite[Cor.\thinspace 2.3.4]{dimca} or  \cite[Thm.\thinspace I.5.2]{McCleary}.
\end{proof}
Obviously Theorem \ref{full1bis} may be extended to other coefficient abelian groups instead of $\mathbb Z$, just quoting \cite{dimca} or \cite{McCleary}

\begin{theorem}\label{ttt1}
Let $Y\subset \mathcal M_n(0)$ be an irreducible compact complex analytic subspace. Then $G(Y)$ is a single point.
\end{theorem}

\begin{proof}
The set $G(Y)$ is an irreducible and compact complex space contained in the affine variety $\mathcal B_n$
 parametrizing all smooth graphs. Thus $G(Y)$ is a point.
\end{proof}

\begin{corollary}\label{ttt2}
Let $Y\subset \mathcal M_n^{\reg}(0)$ be an irreducible compact complex analytic subspace. Then $Y$ is contained in one of the $(2n-1)$-dimensional tori
which are the fibers of the graph map.
\end{corollary}

\begin{proof}
By Theorem \ref{ttt1}, $Y$ is contained in a fiber of the graph map. Since $G^{-1}(G(y))$ is a compact torus of complex dimension $2n-1$, we get the corollary.
\end{proof}

\section{Nontrivial determinants}
Let $\delta \in \mathbb C^*$. In this section we denote the moduli space of rank 2 bundles with $c_2=n$ 
and determinant $\delta$ by $\mathcal M_{n,\delta}$ and set  $\mathcal M_n \ce \mathcal M_{n, \mathcal O_X}$
for the case of trivial determinant.  We have $K_X\cong \mathcal O_X(-2)$.

\begin{remark}\label{wpre1} Let $E$ be a stable (or just a simple) vector bundle or rank $r$ on $X$. Serre-duality gives $h^2(End(E)) =h^0(End(E)\otimes \mathcal O_X(-2)) =0$.
Hence, the local deformation space of $E$ is smooth of dimension $h^1(X,End(E))-1$. 
So, $\mathcal M_{\delta,E}$ is smooth and equidimensional and (Riemann--Roch) $\dim \mathcal M_n=4n$.
\end{remark}

\begin{proposition}\label{wpre2}
For any $\delta \in \mathbb C^*$, and all $n\ge 1$ the complex spaces $\mathcal M_{n,\delta}$ and $\mathcal M_n$ are biholomorphic.
\end{proposition}

\begin{proof}
Both complex spaces are moduli spaces. Hence it is sufficient to prove that $\mathcal M_{n,\delta}$ is a moduli space for rank $2$ stable vector bundles on $X$ with $c_2=n$ and trivial determinant. Fix $c\in \mathbb C^*$ such that $c^2=\delta$. Let $v\colon \mathcal E\to X\times S$ be a family of rank 2 stable vector bundles on $X$ with trivial determinant and $c_2=n$ parametrized by the complex space $S$. Let $\pi_1\colon X\times S\to X$ denote the projection. 
The family $\mathcal E\otimes \pi_1^\ast(\mathcal O_X(c))$ is a family of rank $2$ stable vector bundles on $X$ with $c_2=n$ and determinant isomorphic to $\delta$. Hence there is a holomorphic map 
$f_v\colon S\to \mathcal{M}_{n,\delta}$, which to each $s$ in 
$S$ assigns the bundle $E\otimes \pi_1^\ast(\mathcal O_X(c))$.
 By the definition of coarse moduli space the rule $v\mapsto f_v$ shows that $\mathcal{M}_{n,\delta}$ is (coarse) moduli space for $\mathcal M_n$ and hence 
$\mathcal M_{n,\delta}$ and $\mathcal M_n$ are biholomorphic.
\end{proof}

We conclude that the theorems proved here all have analogous statements 
for the cases of nontrivial determinant. \\

\paragraph{\bf Acknowledgements}  
 E. Ballico is a member of  GNSAGA of INdAM (Italy). 
 E. Gasparim is a senior associate of the Abdus Salam International 
 Centre for Theoretical Physics, Trieste (Italy).


\begin{thebibliography}{99}

\bibitem[AJ]{AJ} M. F. Atiyah, J. D. S. Jones, {\it Topological aspects of Yang-Mills theory}, Comm.
Math. Phys.(2) {\bf 61}  (1978) 97–118,

\bibitem[AF]{af} A. Andreotti, T. Fraenkel, {\it The Lefschetz theorem on hyperplane sections},  
Ann. of Math. (2) {\bf 69} (1957), 713--717.

\bibitem[At]{At}  M. F. Atiyah, {\it Vector bundles over an elliptic curve}, Proc. London Math. Soc. (3) {\bf 7} (1957)  414--452.




\bibitem[BHPV]{BHPV} W. Barth, K. Hulek, C. Peters,  A. Van de Ven, 
{\it Compact complex surfaces}, Springer
Berlin, Heidelberg, 2004.


\bibitem[BHMM]{BHMM} C. P. Boyer, J. C. Hurtubise, B. M.  Mann, R. J. Milgram, 
{\it The topology of instanton moduli spaces. I. The Atiyah–Jones conjecture}, Ann. of Math. (2) {\bf 137} n.3 (1993)
 561--609.

\bibitem[BH]{BH} P. J. Braam, J. Hurtubise,
{\it Instantons on Hopf surfaces and monopoles on solid tori}, J. Reine Angew. Math. {\bf 400} (1989)
146-172.


\bibitem[BMo]{BMo} V. Br\^{\i}nz\u{a}nescu, R. Moraru, {\it Stable bundles on non-K\"ahler elliptic surfaces},
Commun. Math. Phys. {\bf 254}  n.3   (2005) 565--580.


\bibitem[Bu]{Bu} N. P. Buchdahl, 
{\it Hermitian-Einstein connections and stable
vector bundles over compact complex surfaces}, Math. Ann. {\bf 280} (1988) 625--648.

\bibitem[Dc]{tdie} T. tom Dieck, 
{\it Algebraic topology},
EMS Textbk. Math.
European Mathematical Society, Zürich, 2008. 

\bibitem[Di]{dimca} A. Dimca, {\it Sheaves in Topology}, Springer--Verlag, Berlin Heidelberg New York, 2004.

\bibitem[Fi]{fischer} G. Fischer, {\it Complex Analytic Geometry}, Lect. Notes 
 Math. {\bf 538}  Berlin--Heidelberg--New York, 1976.


\bibitem[FMW]{FMW} 
R. Friedman, J. Morgan, and E. Witten, 
{\it Vector bundles over
elliptic fibrations}, J. Algebraic Geom. {\bf 2} (1999) 279-401.


\bibitem[Ga1]{Ga} E. Gasparim,  {\it The Atiyah--Jones conjecture for rational surfaces},
      Advances Math. {\bf 218}, 1027--1050 (2008).

\bibitem[Ga2]{Ga2} E. Gasparim, {\it Connected moduli of instantons on $S^3\times S^1$}, arXiv:2508.19039.


\bibitem[GH]{GH} 
J.P.~Griffiths, J.~Harris,
{\it Principles of algebraic geometry}. Pure and Applied Mathematics. Wiley-Interscience,
 John Wiley \& Sons, New York, 1978.

\bibitem[GKZ]{GKZ} I. M. Gelfand, M. M. Kapranov, A. V. Zelevinsky, 
{\it Discriminant, Resultants, and Multidimensional Determinants}, Modern Birk\"{a}user Classics,
Birk\"{a}user, Boston Basel Berlin, 1994.


\bibitem[Hl]{hl} H. A. Hamm, D. T. L\^{e}, 
{\it The Lefschetz theorem for hyperplane sections},
Springer, Cham, 2020, 491--540.

\bibitem[H]{h} R. Harshorne, {\it Algebraic Geometry}, Springer, Berlin, 1977.


 \bibitem[HM]{HM} J. Hurtubise, R. Milgram,  {\it The Atiyah--Jones conjecture for ruled surfaces},
 J. Reine Angew. Math. {\bf 466} (1995) 111--143.


\bibitem[Jo]{jou} J.-P. Jouanolou, {\it Th\'{e}or\`{e}mes de Bertini et Applications}, Birkh\"{a}user, Basel, 1983.

\bibitem[LT]{LT} M. L\"ubke, A. Teleman,  {\it The Kobayashi–Hitchin correspondence}, River Edge, NJ,
 World Scientific Publishing Co. Inc.,  1995.



\bibitem[MC]{McCleary} J. McCleary, {\it A user's guide to spectral sequences}, Second Edition, Cambridge University Press, Cambridge, 2001.

\bibitem[Mo1]{Mo1} R. Moraru, {\it Integrable systems associated to a Hopf surface}, Canad. J. Math. {\bf 55} n.3
(2003) 609--635.
\bibitem[Mo2]{Mo2} R. Moraru, {\it Stable bundles on Hopf manifolds}, arXiv:0408439.


\bibitem[T]{T} A. Teleman, {\it Moduli spaces of stable bundles on non-K\" ahler elliptic
fibre bundles over curves}, Expo. Math. \textbf{16} (1998) 193--248.


\bibitem[Wh]{whit} G. E. Whitehead, {\it Elements of Homotopy Theory}, Springer-Verlag, New York Heidelberg Berlin, 1978.  

\bibitem[W]{W} E. Witten, {\it Instantons and the large $N=4$ algebra}, J. Phys. A: Math. Theor. {\bf 58} (2025) 035403, arXiv:2407.20964.


\end{thebibliography}
\end{document}